\documentclass{article}

\usepackage{amsmath,latexsym,theorem,amssymb,eufrak}
\input xy
\xyoption{all}
\newtheorem{theorem}{Theorem}
\newtheorem{lemma}{Lemma}
\newtheorem{proposition}{Proposition}

\newtheorem{definition}{Definition}

\newtheorem{corollary}{Corollary}

{ \theorembodyfont{\rmfamily}
  \theoremheaderfont{\scshape}
}
\def\squareforqed{\hbox{\rlap{$\sqcap$}$\sqcup$}}
\def\qed{\ifmmode\squareforqed\else{\unskip\nobreak\hfil
\penalty50\hskip1em\null\nobreak\hfil\squareforqed
\parfillskip=0pt\finalhyphendemerits=0\endgraf}\fi}

\newenvironment{proof}{{\sc Proof:}}{\qed}
\newcommand{\CAT}[1]{\mathbf{#1}}
\newcommand{\SET}{\CAT{Set}}
\newcommand{\CHAUS}{\CAT{CHaus}}
\newcommand{\LIMIT}{\varprojlim}
\newcommand{\COLIMIT}{\varinjlim}

\title{Graphs, Ultrafilters and Colouring}

\author{Felix Dilke\footnote{fdilke@gmail.com}}

\date{}
\begin{document}

\bibliographystyle{plain}
\maketitle

\begin{abstract}
\noindent
Let $\beta$ be the functor from $\SET \rightarrow \CHAUS$ which maps
each discrete set $X$ to its Stone-$\check{\mbox{C}}$ech compactification, the
set $\beta X$ of ultrafilters on $X$.  Every graph $G$ with vertex set
$V$ naturally gives rise to a graph $\beta G$ on the set $\beta V$ of
ultrafilters on $V$.  In what follows, we interrelate the properties of
$G$ and $\beta G$.  Perhaps the most striking result is that $G$ can
be finitely coloured iff $\beta G$ has no loops.
\end{abstract}

\section{Introduction}

\subsection{Notation}
By a \emph{graph} $G$ we mean two parallel arrows $s, t : E \rightarrow V$ in
$\SET$. We think of $V$ as the \emph{vertex set}, $E$ as the \emph{edge set}, and $s(e), t(e)$ as the
\emph{source} and \emph{target} vertices of the edge $e \in E$. 

In particular, each edge is directed, and $G$ may have multiple edges, so this
does not correspond exactly to the usual notion of graph. As will become clear later,
this distinction could be avoided by defining graphs in a slightly more complex way, and
in any case does not matter for any of the results we derive. 

When we want to emphasize this way of defining a graph, we shall call such an arrow-pair a 
\emph{categorical digraph}.

Regarded as an endofunctor of $\SET$, 
$\beta$ clearly preserves the shape of set diagrams, so it maps graphs to graphs. 
Explicitly, given a graph $G = \langle s,t: E \rightarrow V \rangle$, 
we write $\beta G$ for the arrow-pair 
$\langle \beta s, \beta t : \beta E \rightarrow \beta V \rangle$ which is again a graph,
with $\beta E$ and $\beta V$ regarded as ordinary sets.

\subsection{Categorical interpretation of graph properties}

With this formulation, many familiar properties of graphs can be expressed easily in categorical terms,
and some of them turn out to be dual to each other. For example,
\begin{itemize}
\item $G$ has no loops $\iff \LIMIT G = \emptyset$

\item $G$ is weakly connected $\iff \COLIMIT G = 1$,
\end{itemize}
Here the limit is an equalizer, the colimit is a coequalizer, and $0 = \LIMIT \emptyset$ and $1 = \COLIMIT \emptyset$ 
are the initial and terminal objects of $\SET$; so these just say that $G$ has the same limit (resp. colimit) as the empty diagram. 
In particular, weak connectedness of graphs is a dual notion to that of being loop-free.

A \emph{colouring} of $G$ is a map $c:V \rightarrow C$ such that the composite graph 
$cG = \langle cs, ct : E \rightarrow C \rangle$ has no loops, i.e. $\LIMIT cG = \emptyset$. 
(If we require $c$ to be epi, this is dual to the concept of a \emph{spanning set of edges} of $G$.)

We can also construct two new graphs from G: its pullback, which is just its edge graph, and its pushout (less easy to describe).

G has no multiple edges iff $s \times t : E \rightarrow V \times V$ is mono. 
Dually, the coproduct map $s + t : E + E \rightarrow V$ is epi iff G has no isolated points.

\subsection{$\beta G$ is an extension of $G$}
Since $\beta$ is a monad on $\SET$, there are canonical embeddings $\eta(V): V \rightarrow \beta V$ 
and $\eta(E): E \rightarrow \beta E$, which make the following diagrams commute:
\[
\begin{array}{lcl}
\xymatrix
{
	E \ar[r]^{\eta(E)} \ar[d]_{s} & \beta E \ar[d]^{\beta s} \\
	V \ar[r]_{\eta(V)} & \beta V
}
	&\quad&
\xymatrix
{
	E \ar[r]^{\eta(E)} \ar[d]_{t} & \beta E \ar[d]^{\beta t} \\
	V \ar[r]_{\eta(V)} & \beta V
}
\end{array}
\]
The horizontal maps are injective, so $G$ can be embedded as a subgraph: $G \leqslant \beta G$. 
When $G$ is finite, $\beta G=G$; otherwise it is a proper extension.\\
From the general theory of monads (cf.~\cite{adamekj:absconcat}), $\beta E$ and $\beta V$ have the structure of (free) $\beta$-algebras, that is, compact
Hausdorff spaces (cf.~\cite{johnstonep:stospa}), and $\beta G$ can be regarded as a ``topological graph'', whose incidence functions $\beta s, \beta t$ 
act continuously from edges to vertices.

\subsection{More general diagrams}

More generally, for any small category $C$, the functor category $\hat C = \SET^{C^{op}}$ inherits any monad T defined on $\SET$. 
The above is just the special case where $T = \beta$ and $C$ is the two-object category $II = \langle \bullet \rightrightarrows \bullet \rangle$. \\
We can therefore regard $\hat{II}$ as the category of graphs; in particular, since $\hat C$ is a topos (cf.~\cite{moerdijki:shegeo}) there is a natural
concept of exponentiation of graphs. (As shown in \cite{moerdijki:shegeo}, $\hat C$ can be regarded as a cocompletion of $C$, 
which may be an interesting way to look at the theory of graphs in the case $C = II$.) 

To make each $G \in \hat C$ strictly a graph rather than a digraph, we could enforce symmetry of edges by instead setting $C = II'$, where $II'$ is
an enlarged category, exactly like $II$ but with another morphism added to interchange $s$ and $t$. Explicitly, if $s$ and $t$ are the two 
non-identity arrows in $II$, we add an endomorphism $h$ of the second object such that $hs = t, h^{2}=1$. To eliminate loops, one could
go further, introducing a category $II''$ with another morphism from the second object to the first, picking out a loop at each vertex: see
\cite{lawverew:conmat} with its concept of ``irreflexive graph''.
But for our purposes, all these constructions would only complicate the proofs below without adding to their content; we leave $C = II$. \\

We return to discussion of the more general case later.

\subsection{Ultrafilter constructions}
Our arguments will often hinge on whether or not an ultrafilter of a certain type exists. The following basic property of ultrafilters is useful:
\begin{lemma}
For any family of sets $\Theta$, the following are equivalent:
\begin{itemize}

\item There exists an ultrafilter $\xi$ containing all the sets of $\Theta$.

\item $\Theta$ has the finite intersection property: For any sets $A_{i} \in \Theta, 1 \leqslant i \leqslant n$ we have 
$\cap_{i=1}^{n} A_{i} \neq \emptyset$.

\end{itemize}
\end{lemma}
\begin{proof}
See \cite{cohnp:unialg}.
\end{proof}

\ \\
\noindent
We shall also constantly use the following facts:
\begin{itemize}
\item If $\{A_{i}: 1 \leqslant i \leqslant n \}$ is a partition of some set $X$, then any $\upsilon \in \beta X$ contains $A_{i}$ for exactly one $i$. 
\item Given $\upsilon \in \beta X$ and $W \subseteq X$, then: $W$ meets every $V \in \upsilon$ iff $W \in \upsilon$.
\end{itemize}
\noindent
Recall the action of the functor $\beta$ on morphisms: If $f:X \rightarrow Y$ in $\SET$,
the map $\beta f: \beta X \rightarrow \beta Y$ is defined by
\begin{eqnarray}
	V \in \beta f(\upsilon) \iff f^{-1}V \in \upsilon \qquad \forall \ V \subseteq Y, \label{eq:ultmorph}
\end{eqnarray}
for any $\upsilon \in \beta X$.
\ \\
\ \\
We use the following notation for principal ultrafilters, i.e. the images of the canonical embedding map $\eta_X:X \rightarrow \beta X$:
\[	[x] = \eta_X(x) = \{ A \subseteq X : x \in A \} \in \beta X, \qquad x \in X.
\]

\section{Small Relations}
We briefly describe this theory, which is roughly comparable to the Baire theory of 'sets of the
first category' in a suitable topological space (cf.~\cite{rudinw:funana}). It is needed to describe a later
criterion for the existence of multiple edges in $\beta G$.
\ \\
\ \\
A relation $\rho \subseteq X \times Y$ is \emph{small} if relative to it, every set is a finite union of rectangles, i.e. if
\[ 
	\forall Q \subseteq \rho: \exists A_{1}, ..., A_{n} \subseteq X, B_{1}, ..., B_{n} \subseteq Y\ \text{with}\ 
	Q = \rho \cap \bigcup (A_{i} \times B_{i}).
\]		
Smallness is an absolute property of the relation, and does not depend on the rectangle in which it is embedded:
\begin{lemma}
Let $\rho \subseteq X \times Y, X \subseteq X', Y \subseteq Y'$: then $\rho$ is small in $X \times Y$ iff it is small in $X' \times Y'$.
\end{lemma}
\begin{proof}
Immediate.
\end{proof}
\begin{lemma}
The small relations form an ideal: that is,
\begin{itemize}
\item The empty relation is small.
\item Any subrelation of a small relation is small.
\item Any finite union of small relations is small.
\end{itemize}
\end{lemma}
\begin{proof}
Immediate.
\end{proof}
\begin{lemma}
Any finite relation is small.
\end{lemma}
\begin{proof}
Immediate.
\end{proof}
\begin{lemma}
The ``full relation'' $X \times Y$ is small iff either $X$ or $Y$ is finite.
\end{lemma}
\begin{proof}
$\Leftarrow$ Say $|X| < \infty$, $X = \{ x_{i}: 1 \leqslant i \leqslant n \}$. Then for any $Q \subseteq X \times Y$,
define sets $D_{i} \subseteq Y$ by
\[	D_{i} = \{ y \in Y : (x_{i}, y) \in Q \}, \qquad 1 \leqslant i \leqslant n.
\] 
Then
\[	Q = \bigcup_{i=1}^{n} \{ x_{i} \} \times D_{i} \qquad \text{and} \  X \times Y \  \text{is small.}
\]
$\Rightarrow$ It suffices to find a non-small relation between two countably infinite sets. 
We show that the relation $\leqslant$ on $\mathbb{N} \times \mathbb{N}$ is not a union of rectangles. 
Suppose instead that
\[	\leqslant \  = \  \bigcup_{i = 1}^{n} A_{i} \times B_{i}, \qquad A_{i}, B_{i} \subseteq \mathbb{N}.
\]
We can assume no $B_{i}$ is empty. But $a \leqslant b$ for every $a \in A_i$, $b \in B_i$, so each $A_{i}$ must be finite; pick $q \in \mathbb{N} - 
\cup_{i=1}^{n}A_{i}$.
Then $(q, q+1) \notin A_{i} \times B_{i} \ \forall i$ and so $(q, q+1) \notin \ \leqslant$, a contradiction.
\end{proof}
\ \\
Any function $f:X \rightarrow Y$ can be regarded as a relation on $X \times Y$.

\begin{lemma}
Functions and their opposite relations are small.
\end{lemma}
\begin{proof}
Let $f: X \rightarrow Y$ be a function. Then any subrelation $g$ of $f$ is just the restriction of $f$ to some $E \subseteq X$,
and so $g$ = $f \cap (E \times Y)$. Hence $f$ is small. \\
By duality, the same holds for the opposite of $f$.
\end{proof}
\ \\
\noindent
Given two ``composable'' small relations $\rho$ on $X \times Y$ and $\sigma$ on $Y \times Z$, we can ask about their composite 
$\rho \circ \sigma$ on $X \times Z$:
\ \\
\ \\
\emph{Counterexample}: The composite of two small relations need not be small.
\ \\
\ \\
\begin{proof}
Let $X = \omega, Y = 1, Z = \omega$. By the above, the adjacency relations on $K_{(\omega,1)}$ and $K_{(1,\omega)}$ are small, but their composite is all of
$\omega \times \omega$, which is not. 
\end{proof}

\section{Comparing $G$ with $\beta G$}

\subsection{Operational lemmas}
We first prove some technical results that allow us to make deductions about $\beta G$ from $G$ and vice versa.

Write $\rho$ for the adjacency relation in a graph, that is,
\[	x \rho y \iff \exists e \in E : x = s(e), y = t(e).
\]
For subsets $A, B \subseteq V$ we write $A \rho B$ if some $a \rho b$ with $a \in A, b \in B$.
\ \\
We also use $\rho$ for adjacency in $\beta G$: when the distinction is important we use the symbols $\rho_{G}, \rho_{\beta G}$.
\ \\
The following adjacency criterion is fundamental:
\begin{lemma} \label{lem:adjcrit}
Let $\upsilon, \omega \in \beta V$. Then 
\[	\upsilon \rho_{\beta G} \omega \iff A \rho_{G} B \  \forall \ A \in \upsilon, B \in \omega.
\]
\end{lemma}
\begin{proof}
An edge $\xi \in \beta E$ joins $\upsilon$ to $\omega$ just if
\[	(\beta s)(\xi) = \upsilon,\qquad (\beta t)(\xi) = \omega
\]
\[	\text{i.e.} \qquad s^{-1}(A), t^{-1}(B) \in \xi,\qquad \forall \ A \in \upsilon, B \in \omega
\]
We can find such a $\xi$ if and only if these sets have the finite intersection property, that is:
\[	s^{-1}(A) \cap t^{-1}(B) \neq \emptyset,
\]
i.e. just if each $A \rho B$.
\end{proof}

\begin{corollary} \label{cor:adjacent}
$G$ is a full subgraph of $\beta G$: that is, $[x],[y] \in V$ are adjacent in $\beta G$ only if $x,y$ are adjacent in $G$ ($ x, y \in V$). 
\end{corollary}
\begin{proof}
By the above, $[x] \rho_{\beta G} [y]$ just if each $A \rho B$ for $A \in [x], B \in [y]$, 
i.e. for $x \in A, y \in B$. Choosing $A = \{ x \}, B = \{ y \}$ we see this forces 
$\{ x \} \rho_{\beta G} \{ y \}$, i.e. $x \rho_{G} y$.
\end{proof}
\ \\
To relate paths and connectivity between $G$ and $\beta G$, we need some notation: 
For any vertex $x$ and natural number $n$, let $x^{(n)}$ be the set of all vertices which can be reached from $x$ by a path of length $n$.
In particular, $x^{(0)} = \{ x \}$ and $x^{(1)}$ is the set of all \emph{successors} of $x$, i.e. vertices $y$ for which $x \rho y$.
We also write $x^{(-n)} = \{ y : x \in y^{(n)} \}$, the set of all vertices from which $x$ can be reached by a path of length $n$, and we write
\[
	A^{(n)} = \bigcup_{x \in A} x^{(n)}, \qquad A^{(-n)} = \bigcup_{x \in A} x^{(-n)}, \qquad \forall \ A \subseteq V.
\]
The next technical result describes a situation in which a graph can be viewed as a composite of two others, and shows that
this relationship can be lifted from $G$ to $\beta G$.  
\begin{lemma} \label{lem:compthm}
Let $G$, $H$, $K$ be graphs on the same vertex set V. Then
\[	\rho_{G} \rho_{H} = \rho_{K} \qquad \Rightarrow \qquad \rho_{\beta G} \rho_{\beta H} = \rho_{\beta K}. 
\]
\end{lemma}
\begin{proof}
Take $\upsilon, \omega \in \beta V$. Then $\upsilon \rho_{\beta G} \rho_{\beta H} \omega$ just if for some $\xi \in \beta V$,
$\upsilon \rho_{\beta G} \xi$ and $\xi \rho_{\beta H} \omega$. So $\xi$ contains all the sets
\[	A^{(1)} \ \text{and} \ B^{(-1)},\qquad \forall A \in \upsilon, B \in \omega
\]
which says that no finite intersection of these sets is empty. However,
\[	(\cap_{i=1}^{n}A_{i})^{(1)} \subseteq \cap_{i=1}^{n}A_{i}^{(1)}, \qquad 
	(\cap_{i=1}^{n}B_{i})^{(-1)} \subseteq \cap_{i=1}^{n}B_{i}^{(-1)},
\]
(where, implicitly, we use $\rho_{G}$ for calculating the superscript operation for the $A$s, and $\rho_{H}$ for the $B$s).
So the condition is that each $A^{(1)}$ meets each $B^{(-1)}$, i.e. each $A \rho_{G} \rho_{H} B$ ($A \in \upsilon$, $B \in \omega$).
But $\rho_{G} \rho_{H} = \rho_{K}$, so this just says that each $A \rho_{K} B$, i.e. $\upsilon \rho_{K} \omega$. 
\end{proof}
\ \\
An important consequence of this applies to powers of the adjacency relation $\rho_{G}$ : 
For $n > 0$, define $G^{n}$ to be the graph whose edge-set is the set $E^{n}$ of paths of length $n$ in $G$, with the obvious
source and target mappings $s, t: E^{n} \rightarrow V$ taking each path to its initial and final vertex, respectively. Then
$\rho_{G^{n}} = \rho_{G}^{n}$.
\begin{corollary} \label{cor:betapower}
$\rho_{\beta(G^{n})} = \rho_{\beta G}^{n}$.
\end{corollary}
\begin{proof}
Clear when $n = 1$. Since $\rho_{G^{n+1}} = \rho_{G}\rho_{G^{n}}$, we then see inductively 
from $\rho_{\beta(G^{n})} = \rho_{\beta G}^{n}$ that 
$\rho_{\beta(G^{n+1})} = \rho_{\beta G}\rho_{\beta(G^{n})} = \rho_{\beta G}\rho_{\beta G}^{n} = \rho_{\beta G}^{n+1}.$
\end{proof}
\begin{lemma}
Let $\upsilon \in \beta V, x \in V$. Then there is a path of length $n$ in $\beta G$ from $[x]$ to $\upsilon$ iff
the set of $n$th $G$-successors of $x$ belongs to $\upsilon$:
\begin{eqnarray}
	\upsilon \in [x]_{\beta G}^{(n)} \iff x_{G}^{(n)} \in \upsilon.	\label{eq:ultnpath}
\end{eqnarray}
\end{lemma}
\begin{proof}
By induction on $n$. When $n = 1$, we compute
\[	\upsilon \in [x]^{(1)}_{\beta G} \iff [x] \rho_{\beta G} \upsilon \iff A \rho B \  \forall A \in [x], B \in \upsilon \iff \{ x \}\rho B \  \forall B \in \upsilon,
\]
and this last condition says that each $B \in \upsilon$ meets the set $\{y: x \rho y \} = x^{(1)}_{G}$, i.e. $x^{(1)}_{G} \in \upsilon$ as required.\\
For the induction step, suppose (\ref{eq:ultnpath}) holds. Then $\upsilon \in [x]^{(n+1)}$ iff $\omega \rho \upsilon$ for some 
ultrafilter $\omega \in [x]^{(n)}_{\beta G}$,
i.e. if we can find $\omega$ with
\[	x^{(n)} \in \omega \qquad \text{and} \qquad A \rho B \ \forall A \in \omega, B \in \upsilon
\]
Now, $A \rho B$ just if $A$ meets the set $B^{(-1)}$, 
and this happens for every $A \in \omega$ just if $B^{(-1)} \in \omega$.
We can find such an $\omega$ iff each
\[	x^{(n)} \cap \bigcap_{i=1}^{n} B_{i}^{(-1)} \neq \emptyset, \qquad B_{i} \in \upsilon
\]
But $\cap_{i=1}^{n} B_{i} \in \upsilon$ and $(\cap_{i=1}^{n} B_{i})^{(-1)} \subseteq \cap_{i=1}^{n} B_{i}^{(-1)} $; so this holds iff each
\[	x^{(n)} \cap B^{(-1)} \neq \emptyset, \qquad B \in \upsilon,
\]
which says that each $B \in \upsilon$ contains a successor of some $y \in x^{(n)}$, i.e. contains some $z \in x^{(n+1)}$. 
This just says that $x^{(n+1)} \in \upsilon$, as required.
\end{proof}
\subsection{Complete graphs}
We first show that, roughly speaking, $\beta G$ is a complete graph iff $G$ is. 
We need a definition of completeness appropriate to our definition of graph:

\begin{definition}
$G$ is \emph{pseudocomplete} if $s \times t: E \rightarrow V \times V$ is onto.
\end{definition}
\begin{lemma}
$G$ is pseudocomplete $\iff$ $\beta G$ is pseudocomplete.
\end{lemma}
\begin{proof}
$\Leftarrow$ Fix $x, y \in V$. Then $[x], [y]$ are certainly adjacent in $\beta G$, hence by Corollary \ref{cor:adjacent}, 
$x$ and $y$ are already adjacent in $G$.\\
$\Rightarrow$ Fix $\upsilon, \omega$ in $\beta V$. Then we seek $\upsilon \rho \omega$. By Lemma \ref{lem:adjcrit}, 
it suffices to show $A \rho B$ whenever $A \in \upsilon, B \in \omega$. But certainly $A, B$ are nonempty: fix $a \in A, b \in B$.
Then by pseudocompleteness of $G$, we have $a \rho b$ and $A \rho B$.
\end{proof}

\subsection{Invariants}
We next show that several familiar graph invariants are the same for $\beta G$ as for $G$.
(The Greek letters used for these invariants are as in \cite{hararyf:grathe}.)

\begin{proposition}
$G$ and $\beta G$ have the same maximum out-degree; that is, $\Delta(G) = \Delta(\beta G)$,
provided at least one of these is finite.
\end{proposition}
\begin{proof}
Because $G \leqslant \beta G$ is a subgraph, we clearly have $\Delta(G) \leqslant \Delta(\beta G)$. 
Let $d = \Delta(G)$. We have to prove that no ultrafilter $\upsilon \in \beta V$ has more than $d$ $\rho$-successors 
in $\beta G$.\\
Using the axiom of choice, we can find $d$ functions $f_{i}:V \rightarrow V, 1 \leqslant i \leqslant d$ such that
\begin{eqnarray}
	x \rho y \Rightarrow \exists i: f_{i}(x) = y. \qquad \label{enumnbr} 
\end{eqnarray}
Now suppose $\upsilon, \omega \in \beta V$ with $\upsilon \rho \omega$. 
We shall show $\omega = f_{j}(\upsilon)$ for some $j$; this will imply that there can be at most $d$ such $\omega$, and so
$\delta(\upsilon) \leqslant d$, which will complete the proof. \\
By Lemma \ref{lem:adjcrit}, for each $B \in \omega$ we have $A \rho B \ \forall A \in \upsilon$,
so each such $A$ meets the set
\[	B^{(-1)} = \{ x: x \rho b \ \text{for some} \  b \in B \} = \cup_{i=1}^{d}f^{-1}_{i}(B) \ \text{by}\  (\ref{enumnbr}). 
\]
Hence $\cup_{i=1}^{d}f^{-1}_{i}(B) \in \upsilon, \  \forall B \in \omega$. We next claim that for some $j$,
\begin{eqnarray}	f^{-1}_{j}(B) \in \upsilon \qquad \forall B \in \omega. \qquad \label{fjminus1} 
\end{eqnarray}
If not, for each $i$ there is a $B_{i} \in \omega$ with $f^{-1}_{i}(B_{i}) \notin \upsilon$ ; 
letting $B = \cap_{i=1}^{d}B_{i} \in \omega$, we then have each $f^{-1}_{i}(B) \notin \upsilon$ and 
$\cup_{i=1}^{d}f^{-1}_{i}(B) \notin \upsilon$, a contradiction. \\
Hence such a $j$ exists, and from (\ref{fjminus1}) we have $\omega = f_{j}(\upsilon)$ as required.
\end{proof}

\begin{proposition} \label{prop:betagirth}
$G$ and $\beta G$ have the same girth (maximum distance between vertices): $\gamma(G) = \gamma(\beta G)$, 
provided at least one of these is finite.
\end{proposition}
\begin{proof}
Let $\gamma = \gamma(G)$. To show $\gamma \leqslant \gamma(\beta G)$, take $x, y \in V$ of maximum
distance apart, with $d_{G}(x, y) = \gamma$. Then we claim $d_{\beta G}([x], [y]) = \gamma$; for whenever there
is a path of length $n$ from  $[x]$ to $[y]$, applying Corollaries \ref{cor:adjacent} and \ref{cor:betapower},
\[	[x] \rho_{\beta G}^{n} [y] \Rightarrow [x] \rho_{\beta (G^{n})} [y] \Rightarrow
	x \rho_{G^{n}} y \Rightarrow n \geqslant \gamma.
\]
For the other direction, consider $\upsilon, \omega \in \beta V$. 
We claim that $\upsilon \rho^{n} \omega$ for some $n \leqslant \gamma$. 
Otherwise, we can find $A_{n} \in \upsilon, B_{n} \in \omega$ for $0 \leqslant n \leqslant \gamma$
such that no $A_{n} \rho^{n} B_{n}$,
i.e. each $(A_{n} \times B_{n}) \cap \rho^{n} = \emptyset$. But then letting 
\[	A = \cap_{n=0}^{\gamma}A_{n} \in \upsilon, \qquad B = \cap_{n=0}^{\gamma}B_{n} \in \omega 
\]
we have $A \times B \subseteq \cap_{n=0}^{\gamma} A_{n} \times B_{n}$, so
\[	 \emptyset = (A \times B) \cap \bigcup_{n=0}^{\gamma} \rho^{n} = (A \times B) \cap (V \times V) = A \times B,
\]
a contradiction since $A$ and $B$ are both nonempty.      
\end{proof}

\begin{proposition}
$G$ and $\beta G$ have the same chromatic number: $\chi(G) = \chi(\beta G)$, 
provided either of these is finite.
\end{proposition}
\begin{proof}
Any $n$-colouring of $\beta G$ implies one for its subgraph $G$. 
Conversely, suppose $c : V \rightarrow C$ colours $G$, where $|C| = n < \infty$. 
Then $\beta c : \beta V \rightarrow \beta C = C$ since $C$ is finite.
We claim this map colours $\beta G$. \\
To prove this, note that $V$ is partitioned by the sets $c^{-1}(i)$, $i \in C$. 
Hence for any edge $\xi \in \beta E$, the ultrafilter $s(\xi)$ contains $c^{-1}(i)$ for a unique $i$, 
and similarly, $t(\xi)$ contains $c^{-1}(j)$ for a unique $j$. So applying (\ref{eq:ultmorph}),
\[     c(s(\xi)) = i,\qquad c(t(\xi)) = j,
\]
i.e. the source and target of the edge $\xi$ are coloured with $i$ and $j$ respectively.\\
Applying (\ref{eq:ultmorph}) in the other direction,
\[     s^{-1}c^{-1}(i), t^{-1}c^{-1}(j) \in \xi,
\]
which implies that their intersection cannot be $\emptyset$. But for any $k \in C$, there is no edge of $E$ whose 
source and target are both coloured with $k$; that is,
\[     s^{-1}c^{-1}(k) \cap t^{-1}c^{-1}(k) = \emptyset \qquad \forall k \in C.
\]
Hence $i \neq j$ and $\beta c$ is indeed a colouring.
\end{proof}
\ \\
Our results on complete graphs show that the finiteness conditions here are necessary. 
Also, as we shall see, if $\chi(G)$ is infinite then $\beta G$ contains a loop, 
and therefore cannot strictly be coloured at all in our sense. 
\subsection{Connectivity}

\begin{theorem}
$\beta G$ is strongly connected $\iff G$ is strongly connected and has finite girth.
\end{theorem}

\begin{proof}
$\Leftarrow$: The proof of Proposition \ref{prop:betagirth} shows this. 
$\Rightarrow$: Take $x, y \in V$. Then by hypothesis $[x] \rho_{\beta G}^{n} [y]$ for some $n$,
and as in Proposition \ref{prop:betagirth}, this forces $x \rho_{G}^{n} y$, so there is a path from $x$ to $y$ in $G$. 
Hence $G$ is strongly connected. \\
Next, fix $x_{0} \in V$. For any $\upsilon \in \beta V$, by hypothesis some path in $\beta G$ connects $\upsilon$
to $[x_{0}]$, so $\upsilon$ contains a set $x_{0}^{(i)}$. 
Hence no ultrafilter on $V$ can avoid all of these sets, and some finite union of them must therefore be $V$. 
Applying a similar argument to the sets $x_{0}^{(-j)}$, we obtain
\[	V = \bigcup_{i=0}^{m}x_{0}^{(i)} = \bigcup_{j=0}^{n}x_{0}^{(-j)}, \qquad \text{for some}\ m\ \text{and}\ n.
\] 
So for any $u, v \in V$, we can find a path of bounded length from $u$ to $v$ via $x_{0}$, and in fact
$d(u, v) \leqslant m + n.$
Hence $\gamma(G) \leqslant m + n < \infty$.
\end{proof}

\subsection{Multiple Edges}

Call $G$ \emph{proper} if it has no multiple edges; this happens iff the map $s \times t : E \rightarrow V \times V$ is mono. 
To give a condition for $\beta G$ to be proper, we need the theory of ``small relations'' described previously.

\begin{proposition}
$\beta G$ is proper $\iff G$ is proper and $E$ is a small relation on $V \times V$. 
\end{proposition}
\begin{proof}
Certainly since $G \leqslant \beta G$, $G$ is proper if $\beta G$ is. 
Suppose now that $G$ is proper. We regard $E$ as a relation on $V \times V$. Then $\beta G$ is proper just if $\beta s, \beta t$ are
jointly monic on $E$, i.e. every ultrafilter $\xi \in E$
is determined by its images $\beta s(\xi)$ and $\beta t(\xi)$, which are respectively
\[ \{ S \subseteq V : s^{-1}(S) \in \xi \},\qquad \{ T \subseteq V : t^{-1}(T) \in \xi \}
\]
More precisely, $\xi$ is the only ultrafilter on $E$ that contains the sets $s^{-1}(S)$, $t^{-1}(T)$. 
They must therefore generate $\xi$. So every $A \in \xi$ already contains some $s^{-1}(S) \cap t^{-1}(T)$.

For any set $A \subseteq E$, this says that no $\xi \in \beta E$ with $A \in \xi$ can avoid all sets of this form,
and therefore some finite union of them is $A$, say
\[ A = \bigcup_{i \in I}(s^{-1}(S_{i}) \cap t^{-1}(T_{i})) = 
E \cap\bigcup_{i \in I}(s \times t)^{-1}(S_{i} \times T_{i})
\]
where in the last equality $E$ is regarded as a subset of $V \times V$. 
This is precisely the statement that $E$ is a small relation.
\end{proof}
\ \\
The condition that adjacency be a ``small relation'' therefore emerges naturally as a density condition on graphs. 
We call such a graph $sparse$, and briefly note the following, which uses properties of small relations derived above. \ 
\ \\
\ \\
\emph{Counterexample} : A sparse connected graph of finite girth can be infinite. 
\ \\
\ \\
For the complete bipartite graph $G = K_{(1,\omega)}$ is easily seen to have $\rho_{G}$ a subrelation of $1 \times \omega$, 
hence is small, $1$ being finite.
\ \\
\ \\
We also note that sparseness implies the following property, which is in a sense the weakest possible density condition on graphs: 
Call $G$ \emph{weakly sparse} if it has no infinite complete subgraph.
\begin{lemma}
Sparse graphs are weakly sparse.
\end{lemma}
\begin{proof}
Suppose the sparse graph $G$ had an infinite complete subgraph $H$. Then $H$ would also be sparse, and on its vertex set $W$,
the adjacency relation $\rho_{W} = W \times W - \Delta$ would be small, where $\Delta = \{ (w,w): w \in W \}$ is the diagonal. 
But $\Delta$ is a function, and therefore small; hence $W \times W$ is small for infinite W, a contradiction.
\end{proof} 
\subsection{Loops}
We call $G = \langle s,t: E \rightarrow V \rangle$ \emph{loop-free} if $s(e) \neq t(e) \ \forall e \in E$.
\begin{theorem}
$\beta G$ is loop-free iff $G$ is finitely colourable.
\end{theorem}
\begin{proof}
$\Rightarrow$
We try to construct a loop $\xi$ on the edge-set $\beta E$ of $\beta G$. Seeking $s(\xi) = t(\xi)$, we need
\[    s^{-1}(A) \in \xi \iff t^{-1}(A) \in \xi, \qquad \forall A \subseteq V,
\]
i.e. the symmetric difference
\[    s^{-1}(A)\ \Delta \ t^{-1}(A) \notin \xi, \qquad \forall A \subseteq V.
\]
By hypothesis, no such loop $\xi$ exists, so some finite union of these symmetric differences must be the whole set:
\begin{eqnarray}
      E & = & \bigcup_{i=1}^{n}(s^{-1}(A_{i}) \  \Delta \  t^{-1}(A_{i}))  \label{eq:bigcapsym}
\end{eqnarray}
We now construct a finite colouring on $G$, $c: V \rightarrow C$ with $C$ the power set of 
$[1,n] = \{ i \in \mathbb{N}: 1 \leqslant i \leqslant n \}$ as follows:
\[    \forall x \in V,\qquad c(x) = \{ i \in [1,n] : x \in A_{i} \} \in 2^{[1,n]} = C.
\]
To show this is a colouring, consider any edge $e \in E$. By (\ref{eq:bigcapsym}), 
\[    e \in s^{-1}(A_{i}) \  \Delta \  t^{-1}(A_{i})
\]
for some $i$, and so
\[    i \in c(s(e)) \ \Delta \ c(t(e)).
\]
Hence $c(s(e)) \neq c(t(e))$.

$\Leftarrow$ 
Suppose $c: V \rightarrow C$ colours $G$, $|C| < \infty$. 
For all $i \neq j$ in $C$, define
\[    K(i,j) = \{ e \in E : c(s(e)) = i, c(t(e)) = j \} \subseteq E
\]
Then the $K(i,j)$ form a finite partition of $E$. 
For any edge $\xi \in \beta E$ of $\beta G$, the partition property implies that $K(i,j) \in \xi$ for a unique $i$ and $j$. 
Hence $s^{-1}c^{-1}(i)$ and $s^{-1}c^{-1}(j)$, whose intersection is $K(i,j)$, are both $\xi$-sets. Rewriting this in terms of
the source and target vertices $s(\xi), t(\xi) \in \beta V$, we see
\[    c^{-1}(i) \in s(\xi),\qquad c^{-1}(j) \in t(\xi).
\]
Since $c^{-1}(i)$ and $c^{-1}(j)$ are disjoint, this forces $s(\xi) \neq t(\xi)$, i.e. $\xi$ is not a loop.
\end{proof}

\section{Generalizations}
The "categorical digraph" results described here do not seem to depend heavily on the particular category $C = II$ used to model graphs.
We believe that several of them, and in particular the colouring theorem, are not really theorems of graph theory at all
and should have natural generalizations to the case where $C$ is some more general small category. 
\ \\

A starting point might be to replace $II$ by the \emph{simplicial category}
$\Delta$ consisting of all finite nonempty ordinals and order-preserving functions (see \cite{moerdijki:shegeo}). 
This points to a possible theory of colouring for simplicial complexes. 
For now, we just note that truncating $\Delta$ to dimension 2 gives a category very similar to our $II$ and $II''$.

\bibliography{biblio}

\end{document}